\documentclass[english,10pt,a4paper]{article}
\usepackage{graphicx}
\usepackage{babel}
\usepackage{authblk}
\usepackage{tabularx,ragged2e,booktabs,caption}
\usepackage{hyperref}
\usepackage[utf8]{inputenc}
\usepackage{xcolor}
\usepackage{amsfonts}
\usepackage{amssymb}
\usepackage[tbtags]{amsmath}
\usepackage{mathtools}
\usepackage{amsthm}
\usepackage[toc,page]{appendix}
\usepackage{lscape}
\usepackage{wrapfig}
\usepackage{rotating}
\usepackage{float}
\usepackage{natbib} 
\numberwithin{equation}{section}
\allowdisplaybreaks
\newtheorem{thm}{Theorem}
\newtheorem{lem}[thm]{Lemma}
\newtheorem{proposition}{Proposition}
\theoremstyle{remark}
\newtheorem{remark}{Remark}
\newtheorem{corollary}{Corollary}

\title{Verifying the inverse Laplace transform of total expected stock-outs when demand is Poisson}
\author[]{Ettore Settanni}
\affil[]{Institute for Manufacturing, University of Cambridge \\ e-mail: \href{mailto:e.settanni@eng.cam.ac.uk}{e.settanni@eng.cam.ac.uk} }
\date{}  
\begin{document}
	\maketitle
	\begin{abstract}
		This note aims to verify a Laplace transform pair, previously published without proof, concerning the expected stock-out that may occur in a production-inventory systems when demand is Poisson, and the time horizon is finite. Stock-out, or ``backlog'', is understood as the non-negative difference between cumulative demand, a stochastic variable, and production over a finite time interval. The hoped for outcome is greater clarity on how the Laplace transform pair of interest might be arrived at, and a possible improvement in terms of accuracy on the original result.
	\end{abstract}
	\section{Introduction}
	In the context of production-inventory systems, a ``backlog'' or ``stock-out''  is understood to be the non-negative difference between cumulative demand and  production over a time interval. When demand is probabilistic, quantifying stock-out is equivalent to estimating the number of excess demand in a given time interval via a loss function. For the ``textbook'' case in which demand is normally distributed see \citet[Ch.5]{Nahmias.2005}.
	
	Consider the case in which production is just a given quantity $P=1,2,\dots$, and demand is a Poisson random variable with known parameter $\lambda$. In this case, previous work suggests the following relationship between ``cumulative'' expected backlog and its Laplace transform:
	\begin{equation}
		\label{eq1}
		\begin{split}
			\mathcal{L}^{-1}\left\{\left(\frac{\lambda}{\lambda+s}\right)^P\frac{\lambda}{s^3} \right\} & = \frac{\lambda^{P+1}}{(P-1)!}\iiint{e^{-\lambda t}t^{P-1} (dt)^{3}} \\
			&=\frac{\lambda t^{2}}{2}-Pt+\frac{P(P+1)}{2\lambda} - \frac{e^{\lambda t}}{2\lambda} \biggl[P(P+1) \\
			&\textcolor{white}{==}\sum_{j=0}^{P-1} \frac{(\lambda t)^j}{j!} -2P\sum_{j=0}^{P-2} \frac{(\lambda t)^{j+1}}{j!} +\sum_{j=0}^{P-3} \frac{(\lambda t)^{j+2}}{j!} \biggr] 			
		\end{split}
	\end{equation}
	The identity in Eq.~\ref{eq1} combines, \emph{without altering} them, Eq.~2.13 in \citet{Grubb2} and Eqs.~7.4-7.5 in \citet{Grubb1}. Both papers propose that the expected backlog be``cumulated'' so that it may contribute to an aggregate inventory cost function which, when divided by the length of a chosen time horizon, yields an average cost per time unit.
	
	The issue of interest for this note is that neither \citet{Grubb2} nor \citet{Grubb1} provide details on how they arrive at the final result in Eq.~\ref{eq1} in the case of a \textit{finite} time horizon. Instead, both papers cite a common source to which they refer the reader for proof. Yet such source remains, to the best of my knowledge, unpublished.
	
	In the remainder of this note I address the lack of accessible details about how  Eq.~\ref{eq1} may be arrived at by attempting direct proof. I provide a sequence of explicit statements (lemmas) that can be demonstrated directly, and discuss possible improvements in terms of accuracy on the original work. Throughout this note I assume familiarity with key results about Laplace transforms. For the less familiar reader, \citet[][Ch.5]{farlow2012introduction} provides an accessible yet exhaustive introduction to the topic. The specialist reader interested in an overview of applications of the Laplace transform in the context of managing production-inventory systems might benefit from \citet{Grubbstrom.2007}.
	
	\section{Managerial relevance}
	The challenges of estimating expected shortfall has received renewed interest after the disruptive events linked to a recent pandemic. For example \citet{SodhiCOVID2} rely on newsvendor-type shortfall metrics to model trade-offs between strategic stockpiling and alternative policies at the national level. Unlike the normally-distributed ``textbook'' case \citep[e.g.][]{Nahmias.2005}, they assume that the number of people affected by virus-related disasters follow a negative-exponential probability density function.
	
	Shortfalls are of great practical relevance in the context of clinical trial supply chains. It is not uncommon to  overcompensate stock-out risk with high inventory levels to achieve service levels such that clinical trials are carried out effectively. Research on trading-off inventory overage and the risk of a clinical study experiencing shortage, often relies on the Poisson process to model key aspects of clinical trial supply, such as patient recruitment processes \citep{anisimov} and demand of investigational medicinal product in a distribution network \citep[e.g.][]{Fleischhacker}. Yet research rarely grounds these assumptions into extensive analysis of empirical data \citep{settanni2018}.

	\section{Verification strategy}
	\label{replicat}
	The chosen verification strategy is summarised in Fig.~\ref{stratFig}. Both sides of Eq.~\ref{eq1} are replicated, progressing form the transform of backlog, through that of \textit{expected} backlog, to the inverse transform of \textit{cumulative expected} backlog.
	
	For the sake of benchmarking, Table \ref{tab01} shows a comparison between this note's findings and two alternative expressions for the Laplace transform of cumulative expected backlog. Namely, (1) the one given in Eq.\ref{eq1}, reported by previous work, and (2) the result generated by querying the automated computing engine \emph{Wolfram Alpha}, which is freely accessible online \href{https://www.wolframalpha.com}{(www.wolframalpha.com)}. It is worth highlighting that, much like the former, the latter is not accompanied by detailed guidance on how the result is arrived at.

	\begin{sidewaysfigure}[t]
		\includegraphics[scale=0.45]{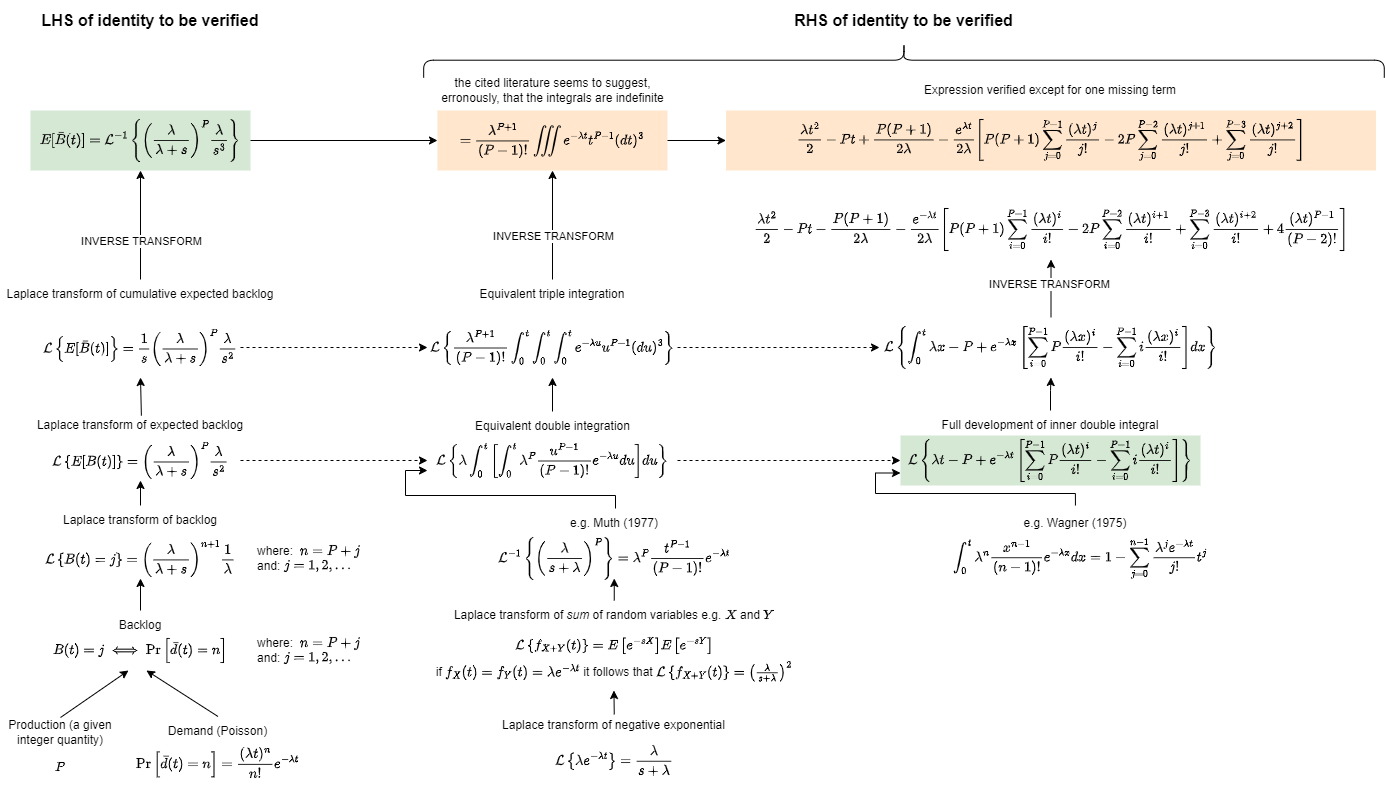}
		\caption{Replication strategy with main results}
		\label{stratFig}
	\end{sidewaysfigure}
	
	\begin{landscape}
	\bgroup
	\def\arraystretch{3}%
	\captionof{table}{Summary of comparative results for $\mathcal{L}^{-1}\left\{\left(\frac{\lambda}{\lambda+s}\right)^P\frac{\lambda}{s^3} \right\}$}
	\label{tab01}		
		\begin{minipage}{\textwidth} 
			\begin{center}
			  \begin{tabular}{|l|l|}	\toprule[1.5pt]	
				Original Equation\footnote{\citet{Grubb2, Grubb1}} & $\frac{\lambda t^{2}}{2}-Pt+\frac{P(P+1)}{2\lambda} - \frac{e^{\lambda t}}{2\lambda} \left[P(P+1) \sum_{j=0}^{P-1} \frac{(\lambda t)^j}{j!} -2P\sum_{j=0}^{P-2} \frac{(\lambda t)^{j+1}}{j!} +\sum_{j=0}^{P-3} \frac{(\lambda t)^{j+2}}{j!} \right]$ \\
				\hline
				Wolfram Alpha\footnote{ \href{https://www.wolframalpha.com}{www.wolframalpha.com}; Please note that this is what I obtain after developing the original result, which is provided in terms of \emph{Gamma} and \emph{incomplete Gamma} functions.} & $\frac{\lambda t^2}{2}-Pt-\frac{P(P+1)}{2\lambda}-\frac{e^{-\lambda t}}{2\lambda}\left[P(P+1)\sum_{i=0}^{P+1}\frac{(\lambda t)^i}{i!} -2P\sum_{i=0}^{P+1}\frac{(\lambda t)^{i+1}}{i!}+\sum_{i=0}^{P+1}\frac{(\lambda t)^{i+2}}{i!}+\frac{(P-1)(\lambda t)^{P+2}-(\lambda t)^{P+3}}{(P+1)!} \right]$   \\
				\hline
				This note & $\frac{\lambda t^2}{2}-Pt-\frac{P(P+1)}{2\lambda}-\frac{e^{-\lambda t}}{2\lambda}\left[P(P+1)\sum_{i=0}^{P-1}\frac{(\lambda t)^i}{i!} -2P\sum_{i=0}^{P-2}\frac{(\lambda t)^{i+1}}{i!}+\sum_{i=0}^{P-3}\frac{(\lambda t)^{i+2}}{i!}-4P\frac{(\lambda t)^{P-1}}{(P-2)!} \right]$ \\
				\hline
			  \end{tabular}			
		    \end{center}
		\end{minipage}
	
	\egroup
	\end{landscape}

	\subsection{Laplace transform of backlog}
	\label{transback}	
	By definition, a backlog $B(t)=j$ occurs if, in a given time interval $t$, cumulative demand exceeds cumulative production by $j$ units. In the case of interest here production $P$ is a given, and demand $\bar{d}(t)$ is stochastic and approximately Poisson-distributed with known parameter $\lambda t$. Then, a backlog $B(t)=j$ can be re-phrased as the probability of exactly $n=P + j$ demand events:
	\begin{align}
		\label{eq3}
			\left[B(t)=j \right] &\equiv \mathrm{Pr}\left[\bar{d}(t)=n \right] \nonumber\\
			& \equiv \frac{(\lambda t)^n}{n!}e^{-\lambda t}, \ n=0,1,2,\dots
	\end{align}
	
	My next	step is to relate the Laplace transform of a backlog to its equivalent formulation in terms of stochastic demand. 
	\begin{lem}
	\label{eq4}
	Let $B(t)=j$ be a non-negative backlog of magnitude $j\in\mathbb{N}$ occurring in a time interval $t$.  Let the demand over the same time interval be $\bar{d}(t) \sim \mathrm{Pois}(\lambda t)$, and production $P\in\mathbb{N}$ be a given. Then,
		\begin{equation}
			\mathcal{L}\left\{ B(t)=j \right\} = \left(\frac{\lambda}{\lambda+s} \right)^{{j+P}}\frac{1}{\lambda+s} \nonumber
		\end{equation}
	\end{lem}	
	\begin{proof} Taking the Laplace transform of both terms in Eq.~\ref{eq3} yields:
		\begin{align*}			
				\mathcal{L}\left\{ B(t)=j \right\} &\equiv \mathcal{L}\left\{\frac{(\lambda t)^{n}}{n!}e^{-\lambda t}  \right\}\\
				&=\int_{0}^{\infty}  e^{-st}\frac{(\lambda t)^{P + j}e^{-\lambda t}}{(n)!} dt \\
				&=\frac{(\lambda t)^{n}}{n!}\int_{0}^{\infty} t^{n}e^{-t(s+\lambda)}dt \\
				&=\frac{(\lambda t)^{n}}{n!}\left[ \frac{n!}{(\lambda+s)^n} \frac{1}{(\lambda+s)} \right] \\
				&=\left(\frac{\lambda}{s+\lambda} \right)^{n+1}\frac{1}{\lambda} \\
				&=\left(\frac{\lambda}{s+\lambda} \right)^{n}\frac{1}{s+\lambda} \nonumber
		\end{align*}
	For $n=j+P$ one obtains the identity in the Lemma
	\end{proof}
	
	\begin{remark}
			In the above I exploit insights from the standard Laplace pair $\mathcal{L}\left\{t^n\right\}=\frac{n!}{s^{n+1}}$ \citep[see e.g.][p.252]{farlow2012introduction}. By definition of Laplace transform, $\mathcal{L}\left\{t^n\right\}=\int_{0}^{\infty}t^ne^{-st}dt$. Integrating by parts yields $\int_{0}^{\infty}t^ne^{-st}dt=\frac{n}{s}\mathcal{L}\left\{t^{n-1}\right\}$. By recursion one obtains	a key result deployed in Lemma~\ref{eq4}		
			\begin{align}
				\int_{0}^{\infty}t^ne^{-st}dt &=\frac{n}{s} \frac{n-1}{s} \frac{n-2}{s} \dots \frac{2}{s} \frac{1}{s} \nonumber \\
				&= \frac{n!}{s^n} \frac{1}{s} \nonumber
			\end{align}
			a result which is linked with the concept of $Gamma$ function \citep[][p.94]{muth1977transform}.
	\end{remark}

	\subsection{Role of the negative exponential density function}
	Although without detail about its derivation, the Laplace transform pair for a known stockout level in Lemma~\ref{eq4} features in \citet[][Eq.2.11]{Grubb2} and in \citet[][Eqs.~5.4 and 3.10]{Grubb1}. In these papers the pair is often expressed in terms of another Laplace transform pair, which involves the probability density function of interarrival times.
	
	The Poisson and negative exponential distributions are related: stating that the probability of $n$ arrivals in any time interval of length $t$ is $\sim \mathrm{Pois}(\lambda t)$ is equivalent to saying that the density function for the time interval $t$ between two consecutive arrivals is exponentially distributed with arrival rate per unit time $\lambda$  \citep[for a proof see e.g., ][pp.860-861]{wagner1975principles}. 
	
	To rephrase Lemma~\ref{eq4} the next step is to work out the Laplace transform pair for the negative exponential density function. This well-known pair is often presented as self-evident -- see e.g., \citet[][p.191]{muth1977transform}, \citet[][p.302]{Grubbstrom.2007} \citet[][p.410]{Grubb1}, and \citet[][p.301]{Grubb2}:
	
	\begin{lem}
		\label{eq5}
		 Let $f(t) = \lambda e^{-\lambda t}$ and $\mathcal{L}\left\{f(t) \right\} = \tilde{f}(s)$. Then,
		 \begin{equation*}
		 	 \tilde{f}(s) = \frac{\lambda}{s + \lambda} 
		 \end{equation*}		 
	\end{lem}
	\begin{proof} Taking the Laplace transform of both sides of $f(t) = \lambda e^{-\lambda t}$ one obtains:
		\begin{align*}
			\mathcal{L}\left\{f(t) \right\} &= \mathcal{L}\left\{ \lambda e^{-\lambda t} \right\} \\
			&= \lambda \int_{0}^{\infty}e^{-t(s+\lambda)} dt\\
			&=\lambda \frac{-e^{-t(s+\lambda)}}{s+\lambda} \Big|_0^\infty \\                        
			&= \frac{\lambda}{s + \lambda} \\
			&= \tilde{f}(s)		\nonumber
		\end{align*}
		which completes the proof
	\end{proof}
	\begin{remark}
		\label{secondexp}
		Based on the above one deduces, after some algebra, that: 
		\begin{equation}			
			1-\mathcal{L}\left\{ \lambda e^{-\lambda t} \right\} = \frac{s}{s+\lambda}  \nonumber
		\end{equation}
	\end{remark}
	
	\begin{corollary}
		\label{eq6}
		The Laplace transform of backlog obtained in Lemma~\ref{eq4} can be re-written combining Lemma~\ref{eq5} and Remark~\ref{secondexp} as follows:
		\begin{align*}
				\mathcal{L}\left\{ \mathrm{Pr}\left[B(t) = j \right] \right\} &= \mathcal{L}\left\{ \lambda e^{-\lambda t} \right\}^{n} \frac{1}{s}\left( 1  - \mathcal{L}\left\{ \lambda e^{-\lambda t} \right\} \right) \\
				&= \tilde{f}(s)^n \frac{1}{s}\left[1  - \tilde{f}(s) \right]		\nonumber
		\end{align*}
	\end{corollary}
	
	One will notice that in Lemma~\ref{eq4} and, equivalently, in Corollary~\ref{eq6} the Laplace transform of the negative exponential density function is raised to the $n$--th power. \citet[][p.196-197]{muth1977transform} shows that this exponentiation is equivalent to transforming the expected value of the sum of $n$ random variables with the same exponential distribution.

	\begin{proposition}[\citealt{muth1977transform}]
		\label{propExp}
		Let $X_1, X_2, \dots, X_n$ be continuous, independent random variables having the same probability density function $f_X(t)=\lambda e^{-\lambda t}$. Let a new random variable $Y$ be the sum $Y=X_1 + X_2 + \dots + X_n$. It follows that its Laplace transform is:
		\begin{align*}
			\mathcal{L}\left\{f_Y(t) \right\} &= E\left[  e^{-sX_1}\right] E\left[e^{-sX_2}\right]\dots E\left[e^{-sX_n}\right] \\
			&= \left( \frac{\lambda}{s + \lambda} \right)^n	
		\end{align*}
	\end{proposition}
	
	The above proposition is illustrated in Appendix \ref{AppendixConvol} for the case $n=2$.

	\subsection{Laplace transform of \textit{expected} backlog}
	\label{expback}
	In this section I move on to the \textit{expected} backlog. The notion of expected value of a random variable \citep{wackerly2014mathematical} is key for demonstrating the following Laplace transform pair:
	\begin{lem}
		\label{eq7}
		Let $E[B(t)]$ be the \textit{expected} backlog occurring in a time interval $t$. As before, let the demand over the same time interval be $\bar{d}(t) \sim \mathrm{Pois}(\lambda t)$, and let production $P\in\mathbb{N}$ be a given. Then,
		\begin{equation}
			\mathcal{L}\left\{ E \left[B(t)\right] \right\} =\left(\frac{\lambda}{\lambda+s} \right)^{P}\frac{\lambda}{s^{2}} 
			\nonumber
		\end{equation}
	\end{lem}	 
	 \begin{proof}	 	
	 	Following the hint in \citet[][Eq.2.12]{Grubb2} and in \citet[][Eq.~5.5]{Grubb1}, let $j=0,1,2,\dots$ be a random variable whose probability density has the Laplace transform given in Lemma~\ref{eq4}. Then,
	 	\begin{align*}	 		
	 		\mathcal{L}\left\{ E \left[B(t)\right] \right\} &= \sum_{j=0}^{\infty}j\left[\left(\frac{\lambda}{\lambda+s} \right)^{P+j}\frac{1}{\lambda+s}\right] \nonumber \\
	 		&=\left(\frac{\lambda}{\lambda+s} \right)^{P}\sum_{j=0}^{\infty}j \ \mathcal{L}\left\{\mathrm{Pr}\left[\bar{d}(t)= j \right] \right\} \nonumber \nonumber\\
	 		&=\left(\frac{\lambda}{\lambda+s} \right)^{P}\mathcal{L}\left\{E\left[\bar{d}(t)\right]\right\} \nonumber\\
	 		&=\left(\frac{\lambda}{\lambda+s} \right)^{P}\mathcal{L}\left\{\lambda t \right\} \nonumber\\
	 		&=\left(\frac{\lambda}{\lambda+s} \right)^{P} \lambda  \int_0^\infty \frac{e^{-st}}{s} dt \nonumber\\
	 		&=\left(\frac{\lambda}{\lambda+s} \right)^{P}\frac{\lambda}{s^{2}}	\nonumber		
	 	\end{align*}
	 which completes the proof
	 \end{proof}
	
	\begin{remark}
	In Lemma~\ref{eq7} I adapt a standard result for the expected value of a Poisson-distributed random variable  \citep[e.g.][pp.134-135]{wackerly2014mathematical} to the specific case of the total arrival of demand events  $\bar{d}(t)$:
		\begin{equation}
			\begin{split}
				E[\bar{d(t)}] &=E \left[\frac{(\lambda t)^{n}e^{-\lambda t}}{n!}\right] \\
				&=\lambda t  \nonumber
			\end{split}		
		\end{equation}
	which is analogous to \citet[][Eq.3.11]{Grubb1} and \citet[Eq.7]{Grubb1998}.
	\end{remark}
	
	\begin{corollary}
		\label{corolEB}
		The result in Lemma~\ref{eq7} can be rewritten using notation introduced in Lemma~\ref{eq5} and Remark~\ref{secondexp} as follows:
		\begin{align*}
			\mathcal{L}\left\{ E \left[B(t)\right] \right\} &=\left(\frac{\lambda}{\lambda+s}\right)^{P}\frac{\lambda}{s^{2}} \\
			&=\left(\cfrac{\lambda}{\lambda + s}\right)^{P+1} \left( s\cfrac{s}{\lambda + s}\right)^{-1} \\
			&=\frac{\tilde{f}(s)^{P+1}}{s(1-\tilde{f}(s))}			
		\end{align*}
	\end{corollary}	
	
	The notation introduced in Corollary~\ref{corolEB}  is frequently adopted in the literature when referring to the Laplace transform pair in Lemma~\ref{eq7} \citep[see e.g.,][]{Grubb2, Grubb1, Grubb1998, Grubb.1999}.
	
	The next step is to show the equivalence between the Laplace transform pair in Lemma.~\ref{eq7} and the inner double integral on the right-hand side of in Eq.~\ref{eq1}: 
	
	\begin{lem}
		\label{eq2}
		Let the Laplace transform of expected backlog be as defined in Lemma~\ref{eq7}. The corresponding time-domain expression for $E \left[B(t)\right]$ is: 
		\begin{align*}
			\mathcal{L}^{-1}\left\{\left(\frac{\lambda}{\lambda+s} \right)^{P}\frac{\lambda}{s^{2}}\right\} &= \lambda t - P + e^{-\lambda t}\left[\sum_{i=0}^{P-1}P\frac{(\lambda t)^i}{i!}-\sum_{i=0}^{P-1}i\frac{(\lambda t)^i}{i!}  \right] 			
		\end{align*}
	\end{lem}
	\begin{proof}
		Recall that \citep[][p.198]{muth1977transform}:
		\begin{equation}
			\label{aux1}
			\mathcal{L}^{-1}\left\{ \left(\frac{\lambda}{s+\lambda}\right)^n  \right\}=\lambda^n\frac{t^{n-1}}{(n-1)!}e^{-\lambda t}
		\end{equation}
		and that \citep[][p.859]{wagner1975principles}:
		\begin{equation}
			\label{aux2}
			\int_{0}^{t}\lambda^n\frac{x^{n-1}}{(n-1)!}e^{-\lambda x}dx=1-\sum_{j=0}^{n-1}\frac{\lambda^j e^{-\lambda t}}{j!}t^{j}
		\end{equation}
		Also recall the standard Laplace transform pair: $\mathcal{L}\left\{\int_{0}^{t}f(u)du\right\}=\frac{1}{s}\mathcal{L}\left\{f(t)\right\}$ \citep[see e.g.][p.65]{muth1977transform}.
		Then, substituting Eq.\ref{aux1} and \ref{aux2}  where appropriate in Lemma~\ref{eq7} yields:
		\begin{align*}
				\left(\frac{\lambda}{\lambda+s} \right)^{P}\frac{\lambda}{s^{2}} &=  \lambda\frac{1}{s^{2}} \mathcal{L}\left\{\lambda^{P}\frac{t^{P-1}}{(P-1)!}e^{-\lambda t}\right\}\\
				&=\mathcal{L}\left\{\lambda\int_{0}^{t}\left[\int_{0}^{t}\lambda^{P}\frac{u^{P-1}}{(P-1)!}e^{-\lambda u} du \right]du\right\}\\
				&=\mathcal{L}\left\{\lambda\int_{0}^{t}\left[ 1-\sum_{j=0}^{P-1}\frac{\lambda^j e^{-\lambda u}}{j!}u^{j} \right] du\right\}\\
				&=\mathcal{L}\biggl\{ \lambda t - \lambda\biggl[\int_{0}^{t}\frac{\lambda^{0}e^{-\lambda u}}{0!}u^{0} du + \int_{0}^{t}\frac{\lambda e^{-\lambda u}}{1!}u du + \\
				& \dots + \int_{0}^{t}\frac{\lambda^{P-1}e^{-\lambda u}}{(P-1)!}u^{P-1}du  \biggr] \biggr\} \\
				&=\mathcal{L}\left\{ \lambda t - P + e^{-\lambda t}\sum_{i=0}^{P-1}\sum_{j=0}^{i}\frac{(\lambda t)^j}{j!} \right\} \\
				&=\mathcal{L}\left\{ \lambda t - P + e^{-\lambda t}\sum_{i=0}^{P-1}(n-i)\frac{(\lambda t)^j}{j!} \right\} \\
				&=\mathcal{L}\left\{ \lambda t - P + e^{-\lambda t}\left[\sum_{i=0}^{P-1}P\frac{(\lambda t)^i}{i!}-\sum_{i=0}^{P-1}i\frac{(\lambda t)^i}{i!}  \right] \right\}		 \nonumber
		\end{align*}
		Taking the inverse transform of both sides produces the sought identity
	\end{proof}
	
	\begin{remark}
		The proof of Lemma~\ref{eq2} features a double sum with linked indices, which was simplified based on the identity $\sum_{i=0}^{n-1}\sum_{j=0}^{i}f(j) =\sum_{i=0}^{n-1}(n-i)f(i)$. For more details on this identify see Appendix \ref{doublesum}.  
	\end{remark}

	\subsection{Laplace transform of \emph{cumulative} expected backlog}
	\label{totexpback}
	Taken together, Lemmas~\ref{eq7} and \ref{eq2} provide important intermediate results towards understanding how \emph{both sides} of Eq.\ref{eq1} might be arrived at.	The Laplace transform of \emph{cumulative expected} backlog, which corresponds to the left-hand side of Eq.~\ref{eq1}, is obtained from Lemma~\ref{eq7} invoking a standard result about the transform of an integral \citep[e.g.][p.258]{farlow2012introduction}:
	\begin{equation}
		\label{eq9}
		\begin{split}
			\mathcal{L}\left\{ E[\bar{B}(t)]\right\} &= \frac{1}{s}\mathcal{L}\left\{ E[B(t)]\right\}\\ 
			&=\frac{1}{s}\left(\frac{\lambda}{\lambda+s} \right)^{P}\frac{\lambda}{s^{2}}
		\end{split}
	\end{equation}

	\begin{lem}
		\label{thirdint}
			Let the image function of the \textit{cumulative expected} backlog in the complex frequency domain be as in Eq.~\ref{eq9}. The original (primitive) function in the time domain is
			\begin{align*}
					\mathcal{L}^{-1} \left\{\left(\frac{\lambda}{\lambda+s}\right)^P\frac{\lambda}{s^3} \right\} &=	\frac{\lambda t^2}{2}-Pt-\frac{P(P+1)}{2\lambda}-\frac{e^{-\lambda t}}{2\lambda}\biggl[P(P+1) \\
					&\sum_{i=0}^{P-1}\frac{(\lambda t)^i}{i!} -2P\sum_{i=0}^{P-2}\frac{(\lambda t)^{i+1}}{i!}+\sum_{i=0}^{P-3}\frac{(\lambda t)^{i+2}}{i!} -4P\frac{(\lambda t)^{P-1}}{(P-2)!} \biggr]
			\end{align*} 		
	\end{lem}
	
	\begin{proof} Invoking the standard Laplace transform pair $\mathcal{L}\left\{\int_{0}^{t}f(u)du\right\}=\frac{1}{s}\mathcal{L}\left\{f(t)\right\}$, I relate Eq.\ref{eq9} and Lemma.~\ref{eq2} as follows:
		\begin{align*}
				 E[\bar{B}(t)] &=\int_{0}^{t}\lambda x - P + e^{-\lambda x}\left[\sum_{i=0}^{P-1}P\frac{(\lambda x)^i}{i!}-\sum_{i=0}^{P-1}i\frac{(\lambda x)^i}{i!}  \right] dx  \\				
				&= \frac{\lambda t^2}{2}-Pt+\frac{P}{\lambda}\sum_{i=0}^{P-1} \left[-\frac{(\lambda t)^i}{i!}e^{-\lambda t} + \int_{0}^{t} \lambda^i 	\frac{x^{i-1}}{(i-1)!} e^{-\lambda x}dx \right]   \\
				&\textcolor{white}{=}-\frac{1}{\lambda}\sum_{i=0}^{P-1} \biggl[-i\frac{(\lambda t)^i}{i!}e^{-\lambda t} + i\int_{0}^{t} \lambda^i \frac{x^{i-1}}{(i-1)!} e^{-\lambda x}dx \biggr]  \\
				&=\frac{\lambda t^2}{2}-Pt-\frac{P}{\lambda}\sum_{i=0}^{P-1}\frac{(\lambda t)^i}{i!}e^{-\lambda t} + \frac{P}{\lambda}\sum_{i=0}^{P-1} \biggl[1-\sum_{j=0}^{i-1}\frac{(\lambda t)^j}{j!}e^{-\lambda t} \biggr] \\
				&\textcolor{white}{=}+\frac{1}{\lambda}\sum_{i=0}^{P-1}i\frac{(\lambda t)^i}{i!}e^{-\lambda t}  - \frac{1}{\lambda}\sum_{i=0}^{P-1} i \biggl[ 1-\sum_{j=0}^{i-1}\frac{(\lambda t)^j}{j!}e^{-\lambda t} \biggr] \\
				&= \frac{\lambda t^2}{2}-Pt-\frac{P(P+1)}{2\lambda}-\frac{e^{-\lambda t}}{\lambda}\biggl[P \sum_{i=0}^{P-1}\frac{(\lambda t)^{i}}{i!} - \sum_{i=0}^{P-2}\frac{(\lambda t)^{i+1}}{i!} \\
				&\textcolor{white}{=}+ P \sum_{i=0}^{P-2}(P-1+i)\frac{(\lambda t)^{i}}{i!}\biggr] + \frac{1}{\lambda}\sum_{i=0}^{P-1}i \left[\sum_{j=0}^{i-1} \frac{(\lambda t)^j}{j!}e^{-\lambda t}  \right] \\
				&=\frac{\lambda t^2}{2}-Pt-\frac{P(P+1)}{2\lambda}-\frac{e^{-\lambda t}}{\lambda}\biggl[P^2\sum_{i=0}^{P-2}\frac{(\lambda t)^{i}}{i!}-(P+1) \\
				&\textcolor{white}{=}\sum_{i=0}^{P-3}\frac{(\lambda t)^{i+1}}{i!} +\frac{(\lambda t)^{P-1}}{(P-1)!}\biggr]+\frac{1}{\lambda}\sum_{i=0}^{P-2}\biggl[\frac{P(P-1)}{2} -\frac{i(i+1)}{2}\biggr]\frac{(\lambda t)^i}{i!}e^{-\lambda t}\\
				&= \frac{\lambda t^2}{2}-Pt-\frac{P(P+1)}{2\lambda}-\frac{e^{-\lambda t}}{2\lambda}\biggl[P(P+1)\sum_{i=0}^{P-2}\frac{(\lambda t)^i}{i!} \\
				&\textcolor{white}{=}-2P\sum_{i=0}^{P-3}\frac{(\lambda t)^{i+1}}{i!}+\sum_{i=0}^{P-4}\frac{(\lambda t)^{i+2}}{i!} +2\frac{(\lambda t)^{P-1}}{(P-1)!} \biggr] \\
		\end{align*}	
	Appendix \ref{longEqApp} provides intermediate steps that are omitted here for brevity. The identity in the Lemma is obtained from the last identity in the proof by leveraging the fact that $\sum_{i=0}^{n}f(i)=\sum_{i=0}^{n-1}f(i)+f(n)$.		
	\end{proof}

	\section{Closing remarks}
	The proposed verification process yields a result which is very close to, but not the same as the alternatives summarised in Tab.~\ref{tab01}. Specifically, Lemma~\ref{thirdint}, differ from Eq.~\ref{eq1}, for one additional term which seems to be absent in the original references \citep[e.g.][Eq.~2.13]{Grubb2}.
	
	It is worth noting, however, the satisfying verification of intermediate results concerning the Laplace transform of \emph{expected backlog} in Section~\ref{expback}. \citet[][p.363]{Grubb.1999} formulates, without proving it, a theorem akin to Lemma~\ref{eq7}. \citet[p.221]{Grubb1998} arrives at a conclusion analogous to Lemma~\ref{eq2}, although by a rather different avenue, which involves complex inversion techniques for finding the inverse Laplace transform. This note contributes a simpler strategy in proving Lemma~\ref{eq2}, which avoids the need for a complex inversion -- a desirable feature according to some literature \citep[see e.g.][p.102]{muth1977transform}.
		
	An incidental benefit of making reference to the standard pair $\mathcal{L}\left\{\int_{0}^{t}f(u)du\right\}=\frac{1}{s}\mathcal{L}\left\{f(t)\right\}$ in the proof of Lemma~\ref{eq2} is that of pinpointing a possible inaccuracy in Eq.\ref{eq1}, which features \textit{indefinite} integrals.
	
%
	
	Whilst it is not the ambition of this note to claim a definitive answer, it is hoped that the process described brings greater clarity on how the Laplace transform pair of interest might be arrived at. To warrant further scrutiny, an outline of the proposed verification strategy was submitted to a specialist online discussion forum where, at the moment of writing, remains amenable to voluntary discussion and input from the public \href{https://math.stackexchange.com/q/4786496}{(math.stackexchange.com/q/4786496)}. 
	
	\appendix
	\begin{appendices}
		\section{Transform of sums of random variables}
		\label{AppendixConvol}
		This section is meant to illustrate Proposition~\ref{propExp}. Following \citet[][p.196-197]{muth1977transform}, I show that (1) the Laplace transform of a random variable $X$ can be expressed in terms of expected value of a function of such variable; and (2) the density function of the sum of two independent random variables corresponds to the \textit{convolution} of the density functions of these random variables.
		In particular, the expected value of a function $g$ of a random variable $X$ is, by definition:
		\begin{equation}
			E[g(X)]=\int_{0}^{\infty}g(t)f_X(t)dt   
		\end{equation}
		
		For $g(X)=e^{-sX}$ we can write
		
		\begin{equation}
			E[e^{-sX}]=\mathcal{L}\left\{f_X(t) \right\} 
		\end{equation} 
		
		If, for example, $X$ and $Y$ are two independent random variables with probability density function $f_{X}(t)$ and $f_{Y}(t)$, respectively, then their sum $Z = X + Y$ has probability density function $f_{X+Y}(t)$, the Laplace transform of the density function of $Y$ is:		
		\begin{align}
				\mathcal{L}\left\{ f_{Z}(t)\right\}=E\left[ e^{-s(X+Y)}\right] &=E\left[  	e^{-sX}\right]E\left[e^{-sY}\right] \nonumber \\
				&= \mathcal{L}\left\{ f_X(t) \right\} \mathcal{L}\left\{ f_Y(t) \right\}=\tilde{f}_{X}\tilde{f}_{Y}		
		\end{align}
		It follows that
		\begin{align}
				f_{X+Y}(t) &= \mathcal{L}^{-1}\left\{ \tilde{f}_{X}(s)\tilde{f}_{Y}(s) \right\} \nonumber \\
				&= f_X(t) \ast f_Y(t)  				
		\end{align}
		
		If $X$ and $Y$ follow the negative exponential density function i.e. $f_X(t)=f_Y(t)=\lambda e^{-\lambda t}$, the Laplace transform of their sum $Z = X + Y$ is as follows:
		\begin{align*}
				\mathcal{L}\left\{f_Z(t) \right\} &= E\left[  e^{-sX}\right]E\left[e^{-sY}\right] \\
				&=\left( \frac{\lambda}{s + \lambda} \right)^2	
				\nonumber
		\end{align*}
		The inverse Laplace transform of a 2-fold convoluted negative exponential density function is
		\begin{align*}
			f_Z(t) & =\mathcal{L}^{-1}\left\{\left( \frac{\lambda}{s + \lambda} \right)\left( \frac{\lambda}{s + \lambda} \right)\right\} \\
			&=\lambda e^{-\lambda t} \ast \lambda e^{-\lambda t} \\
			&=\int_0^t \lambda e^{-\lambda (t-x)} \lambda e^{-\lambda t} dx \\
			&= \lambda^2 e^{-\lambda t} \int_0^t  1 dx \\
			&=\lambda^2 te^{-\lambda t}		
		\end{align*}
		which is an instance of the more general result \citep[][p.198]{muth1977transform}:
		\begin{equation}
			\mathcal{L}^{-1}\left\{ \left(\frac{\lambda}{s+\lambda}\right)^n  \right\}=\lambda^n\frac{t^{n-1}}{(n-1)!}e^{-\lambda t}
		\end{equation}
		
		While the illustrative example involves just two variables, it shows the equivalence between raising the transform of an exponential density function to the $n$--th power and taking the Laplace transform of the expected value of the sum of $n$ random variables that are exponentially distributed with the same parameter $\lambda$. A more general version of this result features in Lemma~\ref{eq4}, the Laplace transform of backlogs.

		\section{Double summations}
		\label{doublesum}
		Throughout the replication process described in Section\ref{replicat}, I've encountered double summations where the inner index is linked to the outer one, and needed simplifying. I was unable to find ready-to-use results about double sums, which forced me to deduce the following equivalences from ``tabulating'' the double summations for small values of $n$:
			\begin{align}
				\sum_{i=0}^{n-1}\sum_{j=0}^{i}f(j) &=\sum_{i=0}^{n-1}(n-i)f(i) \label{A1}\\
				\sum_{i=0}^{n-1}\sum_{j=0}^{i-1}f(j) &=\sum_{i=0}^{n-2}(n-1-i)f(i) \label{A2} \\ 
				\sum_{i=0}^{n-1}i\left[\sum_{j=0}^{i-1}f(j) \right] &=\sum_{j=0}^{n-2} \left[ \frac{n(n-1)}{2}-\frac{j(j+1)}{2}\right]f(j) 
			\end{align}
		For illustration, in Fig.\ref{fig:doublesumtab} I provide an example of how intuition developed from a simple tabulation process for the first equation \ref{A1} mentioned above.
		
		\begin{figure}
			\includegraphics[width=1.1\linewidth]{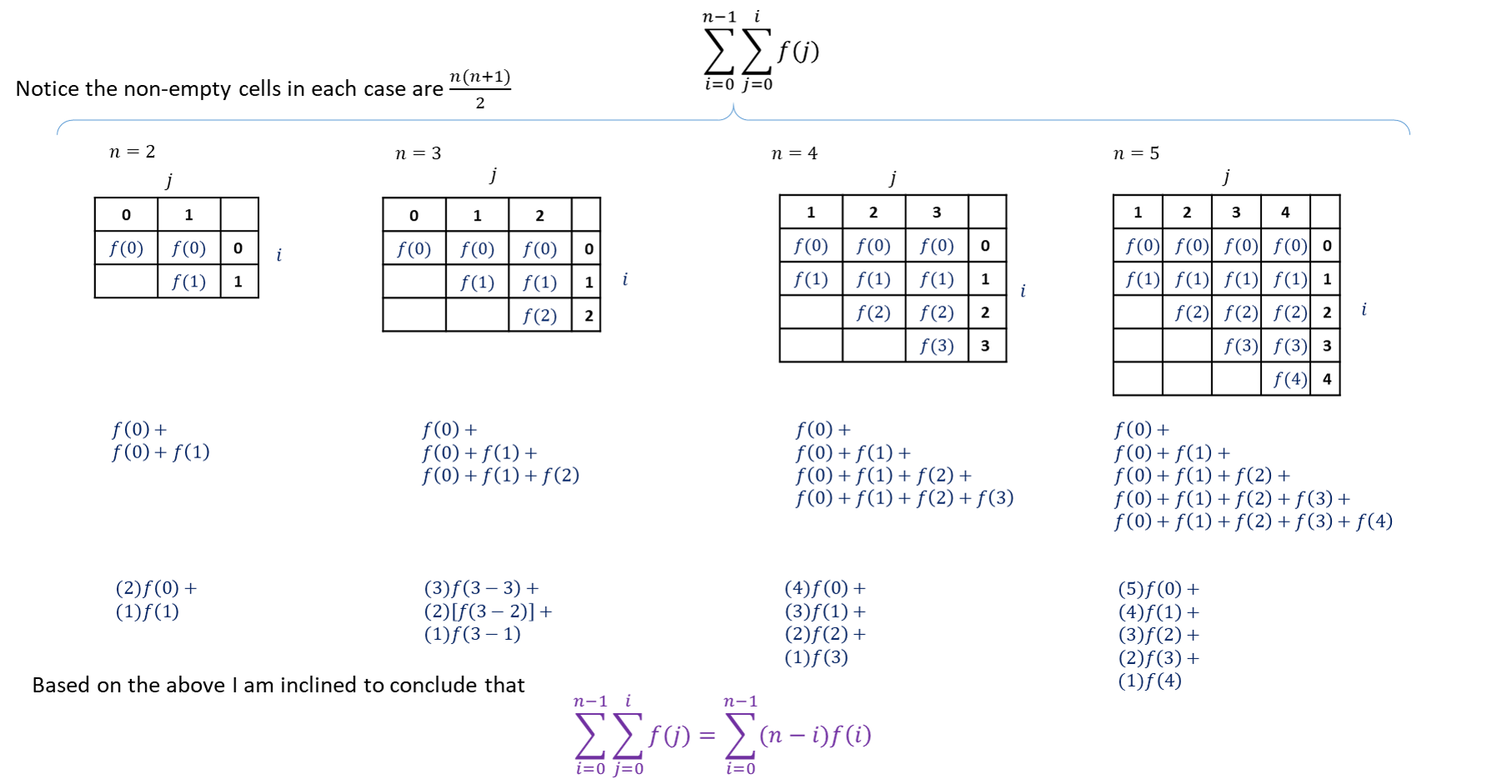}
			\caption[Simplifying double sums with linked indices from tabulations]{simplifying double sums with linked indices by tabulations}
			\label{fig:doublesumtab}
		\end{figure}

		A more rigorous approach is due to \citet{MathSE02}, who addressed a query I have submitted in an online community of experts, with a view to verifying the reliability of the above identities in the absence of other sources: \href{https://math.stackexchange.com/q/4789840/1234572}{math.stackexchange.com/q/4789840/1234572}. 
		
		Specifically, the first two identities are verified as follows:
		\begin{align*}
			{\sum_{i=0}^{n-1}\sum_{j=0}^if(j)}
			&=\sum_{0\leq j\leq i\leq n-1}f(j)\\
			&=\sum_{j=0}^{n-1}f(j)\sum_{i=j}^{n-1}1\\
			&=\sum_{j=0}^{n-1}f(j)\left((n-1)-(j-1)\right)\\
			&=\sum_{j=0}^{n-1}f(j)(n-j)\\
			&\,\,{=\sum_{i=0}^{n-1}f(i)(n-i)}
		\end{align*}
		which agrees with the claim and provides useful insights for verifying Eq.\ref{A2}:
		\begin{align*}
			{\sum_{i=0}^{n-1}\sum_{j=0}^{i-1}f(j)}
			&=\sum_{i=1}^{n-1}\sum_{j=0}^{i-1}f(j)\\
			&=\sum_{i=0}^{n-2}\sum_{j=0}^i f(j)\\
			&\,\,{=\sum_{i=0}^{n-2}f(i)(n-1-i)}
		\end{align*}

		\section{Integral development: cumulative backlog}
		\label{longEqApp}
		In this section I develop in full the in integral in Lemma~\ref{thirdint} to ensure that the steps taken can be  scrutinised. 
		
		Double summations with linked indices were simplified according to the rules shown in Appendix \ref{doublesum}.		The results in Eqs.\ref{aux1} and \ref{aux2} also came in handy, especially at the start of the process. 	
		Finally, on several occasions I applied the index shift rule for summations i.e.,  $\sum_{i=s}^{n}f(i)=\sum_{i=s+p}^{n+p}f(i-p)$, with the understanding that $i=s+p=0$ for $p<0$ and $s=0$.
		
			\begin{landscape}
			\begin{align}
				&\int_{0}^{t}\lambda x - P + e^{-\lambda x}\left[\sum_{i=0}^{P-1}P\frac{(\lambda x)^i}{i!}-\sum_{i=0}^{P-1}i\frac{(\lambda x)^i}{i!}  \right] dx=\frac{\lambda t^2}{2}-Pt+P\sum_{i=0}^{P-1}\frac{\lambda^i}{i!}\int_{0}^{t} x^i e^{-\lambda x}dx - \sum_{i=0}^{P-1}i\frac{\lambda^i}{i!} \int_{0}^{t} x^i e^{-\lambda x}dx \nonumber\\
				&=\frac{\lambda t^2}{2}-Pt+\frac{P}{\lambda}\sum_{i=0}^{P-1} \left[-\frac{(\lambda t)^i}{i!}e^{-\lambda t} + \int_{0}^{t} \lambda^i 	\frac{x^{i-1}}{(i-1)!} e^{-\lambda x}dx \right] - \frac{1}{\lambda}\sum_{i=0}^{P-1} \left[-i\frac{(\lambda t)^i}{i!}e^{-\lambda t} + i\int_{0}^{t} \lambda^i \frac{x^{i-1}}{(i-1)!} e^{-\lambda x}dx \right] \nonumber\\
				&=\frac{\lambda t^2}{2}-Pt-\frac{P}{\lambda}\sum_{i=0}^{P-1}\frac{(\lambda t)^i}{i!}e^{-\lambda t}  +\frac{P}{\lambda}\sum_{i=0}^{P-1} 	\left[1-\sum_{j=0}^{i-1}\frac{(\lambda t)^j}{j!}e^{-\lambda t} \right] + \frac{1}{\lambda}\sum_{i=0}^{P-1}i\frac{(\lambda t)^i}{i!}e^{-\lambda t}  - \frac{1}{\lambda}\sum_{i=0}^{P-1} i \left[ 1-\sum_{j=0}^{i-1}\frac{(\lambda t)^j}{j!}e^{-\lambda t} \right] \nonumber\\			
				&=\frac{\lambda t^2}{2}-Pt-\frac{P}{\lambda}\sum_{i=0}^{P-1}1-\frac{1}{\lambda}\sum_{i=0}^{P-1}i 	-\frac{P}{\lambda}\sum_{i=0}^{P-1}\sum_{j=0}^{i-1}\frac{(\lambda t)^j}{j!}e^{-\lambda t}-\frac{P}{\lambda}\sum_{i=0}^{P-1}\frac{(\lambda t)^i}{i!}e^{-\lambda t}+\frac{1}{\lambda}\sum_{i=0}^{P-1}i\frac{(\lambda t)^i}{i!}e^{-\lambda t}+\frac{1}{\lambda}\sum_{i=0}^{P-1} i \left[\sum_{j=0}^{i-1}\frac{(\lambda t)^j}{j!}e^{-\lambda t} \right] \nonumber\\	
				&=\frac{\lambda t^2}{2}-Pt+\frac{P^2}{\lambda}-\frac{P(P-1)}{2\lambda}-\frac{e^{-\lambda t}}{\lambda}\left[P \sum_{i=0}^{P-1}\frac{(\lambda t)^{i}}{i!}	- \sum_{i=0}^{P-1-1}\frac{(\lambda t)^{i+1}}{(i+1-1)!} + P \sum_{i=0}^{P-2}(P-1+i)\frac{(\lambda t)^{i}}{i!}\right] + \frac{1}{\lambda}\sum_{i=0}^{P-1}i \left[\sum_{j=0}^{i-1} \frac{(\lambda t)^j}{j!}e^{-\lambda t}  \right] \nonumber\\
				&=\frac{\lambda t^2}{2}-Pt+\frac{P(P+1)}{2\lambda}-\frac{e^{-\lambda t}}{\lambda}\left[ P \sum_{i=0}^{P-1}\frac{(\lambda t)^{i}}{i!} - 	\sum_{i=0}^{P-2}\frac{(\lambda t)^{i+1}}{i!} + P(P-1)\sum_{i=0}^{P-2}\frac{(\lambda t)^{i}}{i!} - P\sum_{i=0}^{P-2-1}\frac{(\lambda t)^{i+1}}{(i+1-1)!}\right]+ \frac{1}{\lambda}\sum_{i=0}^{P-1}i \left[\sum_{j=0}^{i-1} \frac{(\lambda t)^j}{j!}e^{-\lambda t}  \right]\nonumber\\				
				&=\frac{\lambda t^2}{2}-Pt+\frac{P(P+1)}{2\lambda}-\frac{e^{-\lambda t}}{\lambda}\left[\left( P \sum_{i=0}^{P-1}\frac{(\lambda t)^{i}}{i!} + 	P(P-1)\sum_{i=0}^{P-2}\frac{(\lambda t)^{i}}{i!} \right) - \left(\sum_{i=0}^{P-2}\frac{(\lambda t)^{i+1}}{i!} + P\sum_{i=0}^{P-3}\frac{(\lambda t)^{i+1}}{i!}	\right)	 \right] + \frac{1}{\lambda}\sum_{i=0}^{P-1}i \left[\sum_{j=0}^{i-1} \frac{(\lambda t)^j}{j!}e^{-\lambda t}\right]\nonumber\\			
				& =\frac{\lambda t^2}{2}-Pt+\frac{P(P+1)}{2\lambda}-\frac{e^{-\lambda t}}{\lambda}\left[ \left( P^2 \sum_{i=0}^{P-2}\frac{(\lambda t)^{i}}{i!}+ 	P\frac{(\lambda t)^{P-1}}{(P-1)!} \right) - \left( (P+1) \sum_{i=0}^{P-3}\frac{(\lambda t)^{i+1}}{i!} + \frac{(\lambda t)^{(P-2)+1}}{(P-2)!} \right) \right] + \frac{1}{\lambda}\sum_{i=0}^{P-1}i \left[\sum_{j=0}^{i-1} \frac{(\lambda t)^j}{j!}e^{-\lambda t}  \right] \nonumber\\			
				&=\frac{\lambda t^2}{2}-Pt+\frac{P(P+1)}{2\lambda}-\frac{e^{-\lambda t}}{\lambda}\left[P^2 \sum_{i=0}^{P-2}\frac{(\lambda t)^{i}}{i!} - 	(P+1)\sum_{i=0}^{P-3}\frac{(\lambda t)^{i+1}}{i!} +P\frac{(\lambda t)^{P-1}}{(P-1)!} - (P-1)\frac{(\lambda t)^{P-1}}{(P-1)!} \right] + \frac{1}{\lambda}\sum_{i=0}^{P-1}i \left[\sum_{j=0}^{i-1} \frac{(\lambda t)^j}{j!}e^{-\lambda t}  \right]\nonumber\\						
				&=\frac{\lambda t^2}{2}-Pt+\frac{P(P+1)}{2\lambda}-\frac{e^{-\lambda t}}{\lambda}\left[P^2\sum_{i=0}^{P-2}\frac{(\lambda 	t)^{i}}{i!}-(P+1)\sum_{i=0}^{P-3}\frac{(\lambda t)^{i+1}}{i!} +\frac{(\lambda t)^{P-1}}{(P-1)!}\right]+\frac{1}{\lambda}\sum_{i=0}^{P-2}\left[\frac{P(P-1)}{2}-\frac{i(i+1)}{2}\right]\frac{(\lambda t)^i}{i!}e^{-\lambda t} \nonumber\\	
				&=\frac{\lambda t^2}{2}-Pt+\frac{P(P+1)}{2\lambda}-\frac{e^{-\lambda t}}{\lambda}\left[P^2\sum_{i=0}^{P-2}\frac{(\lambda t)^{i}}{i!} - 	\frac{P(P-1)}{2}\sum_{i=0}^{P-2}\frac{(\lambda t)^{i}}{i!}+\frac{1}{2}\sum_{i=0}^{P-2}(i+1)\frac{(\lambda t)^{i}}{(i-1)!} - (P+1)\sum_{i=0}^{P-3}\frac{(\lambda t)^{i+1}}{i!}+\frac{(\lambda t)^{P-1}}{(P-1)!} \right] \nonumber\\		
				&=\frac{\lambda t^2}{2}-Pt+\frac{P(P+1)}{2\lambda}-\frac{e^{-\lambda t}}{\lambda}\left[\left( \frac{2P^2-P^2+P}{2}\right)\sum_{i=0}^{P-2}\frac{(\lambda 	t)^{i}}{i!}+\frac{1}{2}\sum_{i=0}^{P-2}i\frac{(\lambda t)^{i}}{(i-1)!}+\frac{1}{2}\sum_{i=0}^{P-2}\frac{(\lambda t)^{i}}{(i-1)!}-(P+1)\sum_{i=0}^{P-3}\frac{(\lambda t)^{i+1}}{i!}+\frac{(\lambda t)^{P-1}}{(P-1)!}  \right] \nonumber \\		
				&=\frac{\lambda t^2}{2}-Pt+\frac{P(P+1)}{2\lambda}-\frac{e^{-\lambda t}}{\lambda}\left[ \frac{P(P+1)}{2}\sum_{i=0}^{P-2}\frac{(\lambda 	t)^{i}}{i!}+\frac{1}{2}\sum_{i=0}^{P-2-1}(i+1)\frac{(\lambda t)^{i+1}}{(i+1-1)!}+\frac{1}{2}\sum_{i=0}^{P-2-1}\frac{(\lambda t)^{i+1}}{(i+1-1)!}-(P+1)\sum_{i=0}^{P-3}\frac{(\lambda t)^{i+1}}{i!}+\frac{(\lambda t)^{P-1}}{(P-1)!} \right] \nonumber \\		
				&=\frac{\lambda t^2}{2}-Pt+\frac{P(P+1)}{2\lambda}-\frac{e^{-\lambda t}}{\lambda}\left[\frac{P(P+1)}{2}\sum_{i=0}^{P-2}\frac{(\lambda 	t)^{i}}{i!}-\frac{2P}{2}\sum_{i=0}^{P-3}\frac{(\lambda t)^{i+1}}{i!}-\frac{1}{2}\sum_{i=0}^{P-3}\frac{(\lambda t)^{i+1}}{i!}+\frac{1}{2}\sum_{i=0}^{P-3}(i+1)\frac{(\lambda t)^{i+1}}{i!}+\frac{(\lambda t)^{P-1}}{(P-1)!} \right] \nonumber\\
				&=\frac{\lambda t^2}{2}-Pt+\frac{P(P+1)}{2\lambda}-\frac{e^{-\lambda t}}{\lambda}\left[\frac{P(P+1)}{2}\sum_{i=0}^{P-2}\frac{(\lambda t)^{i}}{i!}-\frac{2P}{2}\sum_{i=0}^{P-3}\frac{(\lambda t)^{i+1}}{i!}+\frac{1}{2}\sum_{i=0}^{P-3-1}i\frac{(\lambda t)^{i+1+1}}{(i+1-1)!}+\frac{(\lambda t)^{P-1}}{(P-1)!} \right] \nonumber\\		
				&=\frac{\lambda t^2}{2}-Pt-\frac{P(P+1)}{2\lambda}-\frac{e^{-\lambda t}}{\lambda}\left[\frac{P(P+1)}{2}\sum_{i=0}^{P-2}\frac{(\lambda t)^i}{i!}-\frac{2P}{2}\sum_{i=0}^{P-3}\frac{(\lambda t)^{i+1}}{i!}+\frac{1}{2}\sum_{i=0}^{P-4}\frac{(\lambda t)^{i+2}}{i!}+\frac{(\lambda t)^{P-1}}{(P-1)!} \right] \nonumber \\
				&=\frac{\lambda t^2}{2}-Pt-\frac{P(P+1)}{2\lambda}-\frac{e^{-\lambda t}}{2\lambda}\left[P(P+1)\sum_{i=0}^{P-2}\frac{(\lambda t)^i}{i!}-2P\sum_{i=0}^{P-3}\frac{(\lambda t)^{i+1}}{i!}+\sum_{i=0}^{P-4}\frac{(\lambda t)^{i+2}}{i!}+2\frac{(\lambda t)^{P-1}}{(P-1)!} \right] \label{eq10} \\
				&=\frac{\lambda t^2}{2}-Pt-\frac{P(P+1)}{2\lambda}-\frac{e^{-\lambda t}}{2\lambda}\left[ P(P+1)\sum_{i=0}^{P-1}\frac{(\lambda t)^i}{i!}-P(P+1)\frac{(\lambda t)^{P-1}}{(P-1)!}-2P\sum_{i=0}^{P-2}\frac{(\lambda t)^{i+1}}{i!}-2P\frac{(\lambda t)^{P-1}}{(P-2)!}   +\sum_{i=0}^{P-3}\frac{(\lambda t)^{i+2}}{i!}-\frac{(\lambda t)^{P-1}}{(P-3)!} +2\frac{(\lambda t)^{P-1}}{(P-1)!} \right] \nonumber \\
				&=\frac{\lambda t^2}{2}-Pt-\frac{P(P+1)}{2\lambda}-\frac{e^{-\lambda t}}{2\lambda}\left[P(P+1)\sum_{i=0}^{P-1}\frac{(\lambda t)^i}{i!} -2P\sum_{i=0}^{P-2}\frac{(\lambda t)^{i+1}}{i!}+\sum_{i=0}^{P-3}\frac{(\lambda t)^{i+2}}{i!}-4P\frac{(\lambda t)^{P-1}}{(P-2)!} \right] \label{eq10b} 
		\end{align}
		\end{landscape}			
	\end{appendices}

	\bibliography{test-bib}
\end{document}